\documentclass[12pt,reqno]{amsart}
\usepackage{amsmath, amsthm, amssymb}

\advance \topmargin by -\headheight
\advance \topmargin by -\headsep
     
\evensidemargin \oddsidemargin
\newtheorem{theorem}{Theorem}[section]
\newtheorem{lemma}[theorem]{Lemma}

\theoremstyle{definition}

\theoremstyle{remark}

\numberwithin{equation}{section}

\newcommand{\mmod}[1]{\,\,(\text{mod}\,\,#1)}

\def\bfa{{\mathbf a}}
\def\bfb{{\mathbf b}}

\def\bfx{{\mathbf x}}

\def\calA{{\mathcal A}}  

\def\calB{{\mathcal B}}  \def\calBhat{{\widehat \calB}}

\def\calN{{\mathcal N}}

\def\calS{{\mathcal S}}

\def\calZ{{\mathcal Z}}

\def\dbN{{\mathbb N}}
\def\dbR{{\mathbb R}}
\def\dbZ{{\mathbb Z}}\def\dbQ{{\mathbb Q}}

\def\grA{{\mathfrak A}}
\def\grB{{\mathfrak B}}

\def\gre{{\mathfrak e}}\def\grE{{\mathfrak E}}

\def\grJ{{\mathfrak J}}

\def\grm{{\mathfrak m}}\def\grM{{\mathfrak M}}\def\grN{{\mathfrak N}}
\def\grn{{\mathfrak n}}\def\grS{{\mathfrak S}}
\def\grW{{\mathfrak W}}\def\grB{{\mathfrak B}}

\def\grW{{\mathfrak W}}

\def\alp{{\alpha}} 
\def\bet{{\beta}}  
\def\gam{{\gamma}} 
\def\del{{\delta}} \def\Del{{\Delta}}
\def\zet{{\zeta}}  
 
\def\tet{{\theta}} \def\bftet{{\boldsymbol \theta}} \def\Tet{{\Theta}}
\def\kap{{\kappa}}
\def\lam{{\lambda}} \def\Lam{{\Lambda}} 

\def\bfxi{{\boldsymbol \xi}}

\def\sig{{\sigma}}

\def\ome{{\omega}}  
\def\d{{\partial}}
\def\eps{\varepsilon}

\def\le{\leqslant} \def\ge{\geqslant}

\def\d{{\,{\rm d}}}

\begin{document}
\title[Diagonal quartic forms]{Cubic moments of Fourier coefficients\\ and Pairs of diagonal quartic forms}
\author[J\"org Br\"udern]{J\"org Br\"udern}
\address{Mathematisches Institut, Bunsenstrasse 3--5, D-37073 G\"ottingen, Germany}
\email{bruedern@uni-math.gwdg.de}
\author[Trevor D. Wooley]{Trevor D. Wooley}
\address{School of Mathematics, University of Bristol, University Walk, Clifton, Bristol BS8 1TW, United Kingdom}
\email{matdw@bristol.ac.uk}
\thanks{The authors are grateful to the Hausdorff Research Institute for Mathematics in Bonn and the Heilbronn Institute for Mathematics Research in Bristol for excellent working conditions that made the writing of this paper feasible.}
\date{}

\begin{abstract} We establish the non-singular Hasse Principle for pairs of diagonal quartic equations in
 $22$ or more variables. Our methods involve the estimation of a certain entangled two-dimensional
 $21^{\rm{st}}$ moment of quartic smooth Weyl sums via a novel cubic moment of Fourier coefficients.
\end{abstract}
\maketitle

\section{Introduction} A consideration of disjoint systems of diagonal Diophantine equations lends
 credibility to the guiding principle that the number of variables required to solve a system should in general
 grow in proportion to the number of its equations. That such systems are no harder to analyse is made
 evident in an elegant paper of Cook \cite{Coo1972}. There is remarkably little work available in which
 systems of entangled equations have been successfully analysed in circumstances where the average
 number of variables per equation is smaller than that accessible for a single equation. These investigations
 have been limited almost exclusively to equations of degree at most three, and exploit the interaction 
between component equations by means of restricted moments of Fourier coefficients of unconventional
 type (see \cite{BW2007, BW2011, BW2013}). In this paper we add to this small stock of examples,
 analysing pairs of diagonal quartic forms through a novel cubic moment of certain Fourier coefficients of
 arithmetic origin.\par

In order to describe our results, we introduce some notation. When $s$ is a natural number, let $a_j,b_j$
 $(1\le j\le s)$ be fixed rational integers, and consider the pair of simultaneous diagonal quartic equations
\begin{equation}\label{1.1}
a_1x_1^4+a_2x_2^4+\ldots +a_sx_s^4=b_1x_1^4+b_2x_2^4+\ldots +b_sx_s^4=0.
\end{equation}
Given a positive number $P$, we denote by $\calN(P)$ the number of integral solutions $\bfx$ of
 (\ref{1.1}) with $|x_j|\le P$ $(1\le j\le s)$.

\begin{theorem}\label{theorem1.1}
Let $s$ be a natural number with $s\ge 22$. Suppose that $a_i,b_i\in \dbZ$ $(1\le i\le s)$ satisfy the condition that for any pair $(c,d)\in \dbZ^2\setminus \{(0,0)\}$, at least $s-7$ of the numbers $ca_j+db_j$ $(1\le j\le s)$ are non-zero. Then provided that the system (\ref{1.1}) has non-singular real and $p$-adic solutions for each prime number $p$, one has $\calN(P)\gg P^{s-8}$.
\end{theorem}

We pause at this point to discuss the various hypotheses of this theorem. First, the application of
 conventional technology (see \cite{BC1992,Vau1989}) has the potential to establish a conclusion analogous
 to Theorem \ref{theorem1.1} only for $s\ge 24$. Whilst our conclusion saves two variables over this
 bound, one may safely conjecture that subject to appropriate rank conditions the constraint $s\ge 17$
 should suffice. When $s\ge 21$, it follows from work of Godinho \cite{God1992} that the system
 (\ref{1.1}) has non-zero $p$-adic solutions whenever $p>73$. Theorem \ref{theorem1.1} consequently
 implies a Hasse principle for pairs of equations of the shape (\ref{1.1}) in a rather strong form. By
 combining Godinho's work with that of Poehler \cite{Poe2007}, meanwhile, one finds that the $p$-adic
 solubility of the system (\ref{1.1}) is assured for $s\ge 49$.\par

The most novel feature of our proof of Theorem \ref{theorem1.1} involves a consideration of a suitable 
cubic moment of certain Fourier coefficients. For a continuous function $H:\dbR\rightarrow [0,\infty)$ of
 period $1$, let $c(n)=\int_0^1H(\alp)e(-n\alp)\d\alp$, where as usual we write $e(z)$ for $e^{2\pi iz}$. It
 transpires that one may relate the moment
\begin{equation}\label{1.2}
\sum_{n\in \dbZ}|c(n)|^3
\end{equation}
to the correlation
\begin{equation}\label{1.3}
\int_0^1\int_0^1H(\alp)H(\bet)H(\alp+\bet)\d\alp\d\bet .
\end{equation}
We shall take $H(\alp)=\Bigl| \sum_{x\in \calA}e(\alp x^4)\Bigr|^7$, in which $\calA$ is a suitable set of
 smooth numbers. In this case, one may control the moment (\ref{1.2}) by means of the large values
 estimates for Fourier coefficients provided in \cite{KW2010}. The mean value (\ref{1.3}) may be viewed as
 an entangled $21^{\rm st}$-moment of smooth Weyl sums, and it is this that plays the leading role in our
 application of the Hardy-Littlewood (circle) method. We remark that with $H(\alp)$ defined as above, the
 Fourier coefficient $c(n)$ has no direct arithmetic interpretation, despite its arithmetic origin. Indeed, this
 coefficient may be non-zero for arbitrarily large $n$, a complication reflected in technical difficulties that we
 encounter when relating (\ref{1.2}) to (\ref{1.3}) in \S3. Although in principle our methods are applicable
 in wide generality, it would appear that in light of currently available mean value estimates for Weyl sums,
 new ideas are required for their application to Diophantine problems of higher degree.\par

Following some preliminary discussion of smooth Weyl sums, we announce the key $21^{\rm st}$-moment
 estimate in \S2. The aforementioned cubic moment of Fourier coefficients is analysed in \S3, and then
 estimated by means of large values estimates in \S4. In \S5 we shift our focus to preparations for the
 application of the circle method, tackling the minor arcs in the Hardy-Littlewood dissection in \S6, and 
concluding with the major arc analysis in \S7.\par

Our basic parameter is $P$, a sufficiently large positive number. In this paper, implicit constants in
 Vinogradov's notation $\ll$ and $\gg$ may depend on $s$ and $\eps$, as well as ambient coefficients such
 as $\bfa$ and $\bfb$. Whenever $\eps$ appears in a statement, either implicitly or explicitly, we assert that
 the statement holds for each $\eps>0$.

\section{A $21^{\rm st}$-moment of biquadratic Weyl sums} In this section we introduce the novel
 ingredients utilised in our application of the circle method. Before announcing these unconventional mean
 value estimates, however, we must introduce some notation. When $P$ and $R$ are real numbers with 
$1\le R\le P$, we define the set of smooth numbers $\calA(P,R)$ by
$$\calA(P,R)=\{ n\in \dbZ\cap[1,P]:\text{$p$ prime and $p|n\Rightarrow p\le R$}\}.$$
We then define the smooth Weyl sum $h(\alp)=h(\alp;P,R)$ by
$$h(\alp;P,R)=\sum_{x\in \calA(P,R)}e(\alp x^4).$$
It is convenient to refer to an exponent $\Del_t$ as {\it admissible} if there exists a positive number $\eta$
 such that, whenever $1\le R\le P^\eta$, one has
\begin{equation}\label{2.1}
\int_0^1|h(\alp;P,R)|^t\d\alp\ll P^{t-4+\Del_t}.
\end{equation}

\begin{lemma}\label{lemma2.1} The exponents $\Del_7=0.849408$, $\Del_{10}=0.213431$ and
 $\Del_{12}=0$ are admissible.
\end{lemma}

\begin{proof}The desired conclusion follows from \cite[Theorem 2]{BW2000} and the discussion
 surrounding the table of exponents on \cite[page 393]{BW2000}. As explained in the proof of 
\cite[Lemma 2.3]{BW2000}, the upper bound (\ref{2.1}) when $t=12$ is a consequence of 
\cite[Lemma 5.2]{Vau1989}.
\end{proof}

Henceforth, we fix $R=P^\eta$ with $\eta>0$ chosen in accordance with Lemma \ref{lemma2.1} and the 
upper bounds (\ref{2.1}). When $a,b,c,d\in \dbZ$, we define the integral
\begin{equation}\label{2.2}
I(a,b,c,d)=\int_0^1\int_0^1|h(a\alp)h(b\bet)h(c\alp+d\bet)|^7\d\alp\d\bet .
\end{equation}
Our goal in \S\S3 and 4 is the proof of the following upper bound for this integral.

\begin{theorem}\label{theorem2.2} Suppose that $a,b,c,d$ are non-zero integers. Then one has
$$I(a,b,c,d)\ll P^{13+\frac{1}{2}\Del_{10}+\eps}.$$
\end{theorem}

Previous authors would disentangle the mixed moment (\ref{2.2}) via H\"older's inequality to obtain an
 estimate of the shape
$$I(a,b,c,d)\ll \Bigl( \int_0^1|h(A\alp)|^{21/2}\d\alp\Bigr)^2\ll P^{13+\frac{3}{2}\Del_{10}}.$$
The superiority of our new estimate is self-evident. We direct the curious reader to the explanation following
 the statement of \cite[Theorem 3]{BW2007} for a related discussion. By following the argument of the
 proof of \cite[Theorem 4]{BW2007}, presented in \S4 of the latter source, one obtains from Theorem
 \ref{theorem2.2} a superficially more general conclusion of similar type.

\begin{theorem}\label{theorem2.3}
Suppose that $c_i,d_i$ $(1\le i\le 3)$ are integers satisfying
$$(c_1d_2-c_2d_1)(c_1d_3-c_3d_1)(c_2d_3-c_3d_2)\ne 0.$$
Write $\Lam_j=c_j\alp+d_j\bet$ $(j=1,2,3)$. Then whenever $1\le R\le P^\eta$, one has
$$\int_0^1\int_0^1|h(\Lam_1)h(\Lam_2)h(\Lam_3)|^7\d\alp \d\bet \ll 
P^{13+\frac{1}{2}\Del_{10}+\eps}.$$
\end{theorem}

\section{Cubic moments of certain Fourier coefficients}
An inspection of the mean value (\ref{2.2}) suggests that Fourier coefficients associated with $|h(\alp)|^7$
 should play a prominent role in its estimation. The absence of any direct arithmetic interpretation forces us,
 however, to indulge in a far more detailed discussion of these Fourier coefficients than would ordinarily be
 the case. Define
\begin{equation}\label{3.1}
\psi(n)=\int_0^1|h(\alp)|^7e(-n\alp)\d\alp .
\end{equation}

\begin{lemma}\label{lemma3.1} The Fourier expansion
\begin{equation}\label{3.2}
|h(\alp)|^7=\sum_{n\in \dbZ}\psi(n)e(n\alp)
\end{equation}
is uniformly convergent in $\alp$, and one has $\psi(n)\ll P^{15}/n^2$.
\end{lemma}

\begin{proof} We begin by observing that $h(\alp)$ is an analytic function, and hence $|h(\alp)|^7$ is a
 real valued function that is twice continuously differentiable. Note that when $h(\alp)=0$, then the first two
 derivatives of $|h(\alp)|^7$ are also $0$. By writing $|h(\alp)|^7=h(\alp)^{7/2}h(-\alp)^{7/2}$ when
 $h(\alp)\ne 0$, and noting that
$$h'(\alp)=2\pi i\sum_{x\in \calA(P,R)}x^4e(\alp x^4)\ll P^5$$
and
$$h''(\alp)=(2\pi i)^2\sum_{x\in \calA(P,R)}x^8e(\alp x^4)\ll P^9,$$
it follows that, uniformly in $\alp$, one has ${\displaystyle{\frac{\d^2\, \ }{\d\alp^2}}}|h(\alp)|^7\ll 
P^{15}$. Consequently, by integrating by parts, one deduces that
$$\int_0^1|h(\alp)|^7e(-n\alp)\d\alp =\int_0^1\frac{e(-n\alp)}{(2\pi i n)^2}\frac{\d^2\, \ }{\d\alp^2}
|h(\alp)|^7\d\alp \ll \frac{P^{15}}{n^2}.$$
Finally, since $|h(\alp)|^7$ is differentiable and the series on the right hand side of (\ref{3.2}) is absolutely
 convergent, one finds from \cite[Theorem 1.4.2]{DM1972} or \cite[Theorem 8.14, page 60]{Zyg2002}, for example, that the
 Fourier expansion (\ref{3.2}) converges to $|h(\alp)|^7$ uniformly in $\alp$. This completes the proof of the lemma.
\end{proof}

This lemma allows us to replace the mean value (\ref{2.2}) by a cubic moment of $\psi(n)$, truncated both
 in terms of $n$ and the magnitude of $\psi(n)$.

\begin{lemma}\label{lemma3.2} When $a,b,c,d$ are non-zero integers, one has
$$I(a,b,c,d)\ll P^9+\sum_{\substack{|n|<P^9\\ |\psi(n)|>1}}|\psi(n)|^3.$$
\end{lemma}

\begin{proof} Our initial step is to accommodate non-zero integral coefficients $l$ by extracting from
 (\ref{3.2}) the relation
$$|h(l\alp)|^7=\sum_{n\in \dbZ}\psi(n)e(nl\alp)=\sum_{m\in \dbZ}\psi_l(m)e(m\alp),$$
where
\begin{equation}\label{3.3}
\psi_l(m)=\begin{cases} \psi(m/l),&\text{when $l|m$,}\\
0,&\text{otherwise.}\end{cases}
\end{equation}
Using the uniform convergence of the Fourier expansions of $|h(a\alp)|^7$, $|h(b\bet)|^7$ and 
$|h(c\alp+d\bet)|^7$, one finds from (\ref{2.2}) that $I(a,b,c,d)$ is equal to
$$\sum_{n_1,n_2,n_3\in \dbZ}\int_0^1\int_0^1 \psi_a(n_1)\psi_b(n_2)\psi(n_3)
e(n_1\alp+n_2\bet-n_3(c\alp+d\bet))\d\alp\d\bet .$$
By orthogonality, the double integral here can be non-zero only when $n_1=cn_3$ and $n_2=dn_3$,
 whence
$$I(a,b,c,d)=\sum_{n\in \dbZ}\psi_a(cn)\psi_b(dn)\psi(n).$$
Thus, on noting that (\ref{3.3}) delivers the estimate
$$\sum_{n\in \dbZ}|\psi_a(cn)|^3\le \sum_{m\in \dbZ}|\psi_a(m)|^3=\sum_{k\in \dbZ}|\psi(k)|^3,$$
we deduce by H\"older's inequality and symmetry that
$$I(a,b,c,d)\le \sum_{n\in \dbZ}|\psi(n)|^3.$$
We therefore conclude from Lemma \ref{lemma3.1} that
$$I(a,b,c,d)\ll \sum_{|n|\le P^9}|\psi(n)|^3+\sum_{|n|>P^9}(P^{15}/n^2)^3=
\sum_{\substack{|n|\le P^9\\ |\psi(n)|>1}}|\psi(n)|^3+O(P^9),$$
and the proof of the lemma is complete.
\end{proof}

A dyadic dissection simplifies the discussion to come. When $T>0$, write
$$M(T)=\sum_{\substack{|n|\le P^9\\ T<|\psi(n)|\le 2T}}|\psi(n)|^3.$$
By applying the triangle inequality to (\ref{3.1}), the estimates available from Lemma \ref{lemma2.1}
 furnish the bound $\psi(n)\le \psi(0)\ll P^{3+\Del_7}$, and hence
$$\sum_{\substack{|n|\le P^9\\ |\psi(n)|>1}}|\psi(n)|^3\le 
\sum_{\substack{l=0\\ 2^l\le P^{3+\Del_7}\log P}}^\infty M(2^l).$$
Consequently, for some positive number $T$ with $1\le T\le P^{3+\Del_7}\log P$, we may conclude from 
Lemma \ref{lemma3.2} that
\begin{equation}\label{3.4}
I(a,b,c,d)\ll P^9+(\log P)M(T).
\end{equation}

\section{The arithmetic harmonic analysis}
The moment has come to deliver the proof of Theorem \ref{theorem2.2}. We may suppose that $\eta>0$
 is small enough that the estimates implicit in (\ref{2.1}) hold. We bound $M(T)$ when 
$1\le T\le P^{3+\Del_7}\log P$. Define $\calZ$ to be the set of integers $n$ with $|n|\le P^9$ such that
 $T<|\psi(n)|\le 2T$, and write $Z=\text{card}(\calZ)$. For each $n\in \calZ$, we take $\ome_n=1$ when $\psi(n)>0$, and $\ome_n=-1$ when $\psi(n)<0$, and then put
$$K(\alp)=\sum_{n\in \calZ}\ome_ne(-n\alp).$$
Then we find from (\ref{3.1}) that
\begin{equation}\label{4.1}
\int_0^1|h(\alp)|^7K(\alp)\d\alp =\sum_{n\in \calZ}\ome_n\int_0^1|h(\alp)|^7e(-n\alp)
\d\alp =\sum_{n\in \calZ}|\psi(n)|>TZ.
\end{equation}
We bound $M(T)$ by estimating the integral on the left hand side of (\ref{4.1}), controlling the frequency of large Fourier
 coefficients $\psi(n)$.\par

Before proceeding further, we recall that the estimate
\begin{equation}\label{4.2}
\int_0^1|h(\alp)^4K(\alp)^2|\d\alp \ll P^3Z+P^{2+\eps}Z^{3/2}
\end{equation}
is an immediate consequence of \cite[Lemma 2.1]{KW2010}.

\begin{lemma}\label{lemma4.1} One has the estimates
$$Z\ll P^{\frac{28}{3}+\frac{1}{3}\Del_{10}}T^{-2}+P^{13+\frac{1}{2}\Del_{10}+\eps}T^{-3}$$
and
$$Z\ll P^{9+\Del_{10}}T^{-2}+P^{16+2\Del_{10}+\eps}T^{-4}.$$
\end{lemma}

\begin{proof} An application of H\"older's inequality shows that
\begin{align*}
\int_0^1|h(\alp)|^7K(\alp)\d\alp\le &\, \Bigl( \int_0^1|h(\alp)^4K(\alp)^2|\d\alp \Bigr)^{1/3}
\Bigl( \int_0^1 |K(\alp)|^2\d\alp \Bigr)^{1/6}\\
&\, \times \Bigl( \int_0^1|h(\alp)|^{10}\d\alp \Bigr)^{1/6} \Bigl( \int_0^1|h(\alp)|^{12}\d\alp
 \Bigr)^{1/3}.
\end{align*}
Recalling Lemma \ref{lemma2.1}, Parseval's identity and (\ref{4.2}), we deduce from (\ref{4.1}) that
\begin{align*}
TZ&\ll (P^3Z+P^{2+\eps}Z^{3/2})^{1/3}(Z)^{1/6}(P^{6+\Del_{10}})^{1/6}(P^8)^{1/3}\\
&\ll P^{\frac{14}{3}+\frac{1}{6}\Del_{10}}Z^{1/2}+
P^{\frac{13}{3}+\frac{1}{6}\Del_{10}+\eps}Z^{2/3},
\end{align*}
and the first of the claimed estimates follows by disentangling this inequality.\par

On the other hand, also by H\"older's inequality, one has similarly
\begin{align*}
\int_0^1|h(\alp)|^7K(\alp)\d\alp&\le \Bigl( \int_0^1|h(\alp)^4K(\alp)^2|\d\alp \Bigr)^{1/2}
\Bigl( \int_0^1|h(\alp)|^{10}\d\alp \Bigr)^{1/2}\\
&\ll (P^3Z+P^{2+\eps}Z^{3/2})^{1/2}(P^{6+\Del_{10}})^{1/2},
\end{align*}
whence
$$TZ\ll P^{\frac{9}{2}+\frac{1}{2}\Del_{10}}Z^{1/2}+P^{4+\frac{1}{2}\Del_{10}+\eps}Z^{3/4},$$
and the second conclusion follows by further disentangling.
\end{proof}

We are now equipped to establish Theorem \ref{theorem2.2}. On the one hand, if one has 
$T\le P^{3+\frac{3}{2}\Del_{10}}$, then one finds from the first estimate of Lemma \ref{lemma4.1} that
$$M(T)\ll ZT^3\ll P^{\frac{28}{3}+\frac{1}{3}\Del_{10}}T+P^{13+\frac{1}{2}\Del_{10}+\eps}
\ll P^{13+\frac{1}{2}\Del_{10}+\eps}.$$
On the other hand, when $P^{3+\frac{3}{2}\Del_{10}}<T\le P^{3+\Del_7}\log P$, one finds instead from the second estimate of Lemma \ref{lemma4.1} that
$$ZT^3\ll P^{9+\Del_{10}}T+P^{16+2\Del_{10}+\eps}T^{-1}
\ll P^{12+\Del_7+\Del_{10}+\eps}+P^{13+\frac{1}{2}\Del_{10}+\eps}.$$
Thus, on recalling (\ref{3.4}), we deduce that
$$I(a,b,c,d)\ll P^9+(\log P)P^{13+\frac{1}{2}\Del_{10}+\eps}\ll 
P^{13+\frac{1}{2}\Del_{10}+2\eps}.$$
This concludes the proof of Theorem \ref{theorem2.2}.

\section{Preparations for the circle method}
We suppose that the hypotheses of the statement of Theorem \ref{theorem1.1} are satisfied, so in particular
 $s\ge 22$. With the pairs $(a_j,b_j)\in \dbZ^2\setminus \{(0,0)\}$, we associate the binary forms
\begin{equation}\label{5.1}
\Lam_j=a_j\alp+b_j\bet\quad (1\le j\le s),
\end{equation}
and the two linear forms $L_1(\bftet)$ and $L_2(\bftet)$ defined for $\bftet\in \dbR^s$ by
\begin{equation}\label{5.2}
L_1(\bftet)=\sum_{j=1}^sa_j\tet_j\quad \text{and}\quad L_2(\bftet)=\sum_{j=1}^sb_j\tet_j.
\end{equation}
We describe two forms $\Lam_i$ and $\Lam_j$ as {\it equivalent} when there exists a non-zero rational
 number $\lam$ with $\Lam_i=\lam \Lam_j$. This notion defines an equivalence relation, and we refer to
 the number of elements in the equivalence class containing the form $\Lam_j$ as its {\it multiplicity}.\par

The hypotheses of Theorem \ref{theorem1.1} ensure that there is a non-singular real solution of the system
 (\ref{1.1}). By invoking homogeneity, therefore, one finds that there exists a real solution $\bfx=\bftet$ in
 $[0,1)^s$ satisfying the property that for some indices $i$ and $j$ with $1\le i<j\le s$, one has
$$\det \left( \begin{matrix}4a_i\tet_i^3&4a_j\tet_j^3\\ 4b_i\tet_i^3&4b_j\tet_j^3\end{matrix} \right)
\ne 0.$$
By relabelling variables if necessary, there is no loss of generality in supposing that $i=1$ and $j=2$, and
 by taking suitable linearly independent linear combinations of the equations comprising (\ref{1.1}), we may
 suppose further that $a_2=b_1=0$ and $a_1b_2\ne 0$. Thus, since $\bftet$ is a non-singular solution, we
 have $\tet_1\tet_2\ne 0$, and then there is no loss of generality in supposing also that $\tet_1>0$ and
 $\tet_2>0$. An application of the inverse function theorem consequently confirms that whenever $\del>0$
 is sufficiently small, the simultaneous equations
$$a_1x_1^4=-\sum_{i=3}^sa_i(\tet_i+\del)^4\quad \text{and}\quad
 b_2x_2^4=-\sum_{i=3}^sb_i(\tet_i+\del)^4$$
remain soluble for $x_1$ and $x_2$, with $x_1>0$ and $x_2>0$. In this way we see that the system
 (\ref{1.1}) possesses a non-singular real solution $\bftet$ satisfying $\bftet\in (0,1)^s$.\par

For any pair $(c,d)\in \dbZ^2\setminus \{(0,0)\}$, the linear form $cL_1(\bfxi)+dL_2(\bfxi)$ necessarily
 possesses at least $s-7$ non-zero coefficients. By choosing an appropriate subset $\calS$ of 
$\{2,\ldots ,s\}$ with $\text{card}(\calS)=21$, therefore, we may ensure that the forms $\Lam_j$ with
 $j\in \calS$ have multiplicity at most $7$. Suppose that these $21$ forms fall into $t$ equivalence classes,
 and that the multiplicities of the representatives of these classes are $r_1,\ldots ,r_t$. Then we may
 suppose that
\begin{equation}\label{5.3}
7\ge r_1\ge r_2\ge \ldots \ge r_t\quad \text{and}\quad r_1+\ldots +r_t=21,
\end{equation}
and hence also that $t\ge 3$. We relabel variables in (\ref{1.1}), and likewise in (\ref{5.1}) and (\ref{5.2}),
 so that $\calS$ becomes $\{2,\ldots ,22\}$, and for $1\le i\le t$ the  linear form $\Lam_{i+1}$ is in the
 $i$th equivalence class counted by $r_i$. We now fix a non-singular real solution $\bftet\in (0,1)^s$ of
 (\ref{1.1}) and a real number $\del$ with $0<\del<\tet_1$. In addition, we fix $\eta>0$ to be sufficiently
 small in the context of Lemma \ref{lemma2.1}.\par

Next define the generating functions
$$g(\alp)=\sum_{\del P<x\le P}e(\alp x^4),\quad H_0(\alp,\bet)=\prod_{j=2}^{22}h(\Lam_j),\quad
 H(\alp,\bet)=\prod_{j=2}^sh(\Lam_j).$$
Then by orthogonality, one has
$$\calN(P)\ge \int_0^1\int_0^1 g(\Lam_1)H(\alp,\bet)\d\alp\d\bet.$$
In order to define the Hardy-Littlewood dissection underlying our argument, we put $Q=(\log P)^{1/100}$,
 and when $a,b\in \dbZ$ and $q\in \dbN$ we define
$$\grN(q,a,b)=\{(\alp,\bet)\in [0,1)^2:|\alp-a/q|\le QP^{-4}\quad \text{and}\quad 
|\bet-b/q|\le QP^{-4}\}.$$
We then take $\grN$ to be the union of the boxes $\grN(q,a,b)$ with $0\le a,b\le q\le Q$ and $(q,a,b)=1$. 
Finally, we put $\grn=[0,1)^2\setminus \grN$.\par

The contribution of the major arcs $\grN$ in this dissection satisfies
\begin{equation}\label{5.4}
\iint_\grN g(\Lam_1)H(\alp,\bet)\d\alp\d\bet \gg P^{s-8},
\end{equation}
a fact we confirm in \S7. Meanwhile, in \S6 we show that
\begin{equation}\label{5.5}
\iint_\grn g(\Lam_1)H(\alp,\bet)\d\alp\d\bet \ll P^{s-8}(\log \log P)^{-1}.
\end{equation}
The desired conclusion $\calN(P)\gg P^{s-8}$ is immediate from (\ref{5.4}) and (\ref{5.5}) on noting that
 $[0,1)^2$ is the disjoint union of $\grN$ and $\grn$.

\section{The minor arc treatment}
The analysis of the minor arc contribution proceeds in two phases, one dominated by the use of Weyl's inequality, and a second
 in which pruning methods are deployed. We begin preparatory work for the first stage by deriving a consequence of Theorem 
\ref{theorem2.3}.

\begin{lemma}\label{lemma6.1} One has
$$\int_0^1\int_0^1|H_0(\alp,\bet)|\d\alp\d\bet \ll P^{13+\frac{1}{2}\Del_{10}+\eps}.$$
\end{lemma}

\begin{proof} Recall the discussion of the multiplicities associated with $\Lam_2,\ldots ,\Lam_t$, and in
 particular the hypothesis (\ref{5.3}). By applying \cite[Lemma 5]{BW2007}, just as in the deduction of 
\cite[equation (4.5)]{BW2007}, one finds that
\begin{equation}\label{6.1}
\int_0^1\int_0^1|H_0(\alp,\bet)|\d\alp\d\bet \ll \int_0^1\int_0^1 h_2^{r_2}\ldots h_t^{r_t}
\d\alp\d\bet ,
\end{equation}
where we have abbreviated $|h(\Lam_j)|$ to $h_j$. Let $\nu$ be a non-negative integer, and suppose that
 $r_{t-1}=r_t+\nu<7$. Then we may apply the argument of the proof of \cite[Lemma 6]{BW2007}
 following equation (4.5) therein to obtain a bound of the shape (\ref{6.1}), in which $r_{t-1}$ and
 $r_t=r_{t-1}-\nu$ are replaced by $r_{t-1}+1$ and $r_t-1$, respectively, or else by $r_{t-1}-\nu-1$ and
 $r_t+\nu+1$. By relabelling if necessary, we thus derive a bound of the shape (\ref{6.1}), subject to
 (\ref{5.3}), wherein either $r_t$ is reduced, or else $t$ is reduced. By repeating this process, therefore, we
 ultimately arrive at the situation in which $t=3$ and $(r_1,r_2,r_3)=(7,7,7)$. From here the desired
 estimate follows from (\ref{6.1}) and Theorem \ref{theorem2.3}.
\end{proof}

In order to prepare for the pruning process, we introduce a conventional set of one-dimensional major arcs. Define $\grM$ to 
be the union of the intervals
$$\grM(q,a)=\{ \alp\in [0,1):|q\alp-a|\le P^{-7/2}\}$$
with $0\le a\le q\le P^{1/2}$ and $(a,q)=1$, and put $\grm=[0,1)\setminus \grM$. In addition, when $\lam,A\in \dbR$,
 define the mean value $J(\lam)=J(\lam;A)$ by putting
\begin{equation}\label{6.2}
J(\lam;A)=\int_\grM |g(a_1\tet)|^{9/4}|h(A\tet+\lam)|^4\d\tet .
\end{equation}

\begin{lemma}\label{lemma6.2} For each $A\in \dbQ\setminus\{0\}$, one has $\sup_{\lam\in \dbR}J(\lam;A)\ll P^{9/4}$.
\end{lemma}

\begin{proof} Suppose that $\lam\in \dbR$. Write $A=B/S$ with $B\in \dbZ\setminus \{0\}$, $S\in \dbN$ and $(B,S)=1$. We
 define the modified set of major arcs $\grW$ by putting
$$\grW=\{\bet \in [0,1):S\bet \in \grM\}.$$
Then a change of variable yields
\begin{equation}\label{6.3}
J(\lam)=S\int_\grW |g(a_1S\bet)|^{9/4}|h(B\bet +\lam)|^4\d\bet .
\end{equation}
It follows from the definition of $\grM$ that for each $\bet \in \grW$, there exist $c\in \dbZ$ and $r\in \dbN$ with 
$0\le c\le r\le P^{1/2}$, $(c,r)=1$ and $|S\bet r-c|\le P^{-7/2}$. Thus there exist $a\in \dbZ$ and $q\in \dbN$ with 
$0\le a\le q\le SP^{1/2}$, $(a,q)=1$ and $|q\bet -a|\le P^{-7/2}$. Next, we define $\kap(q)$ to be the multiplicative function
 defined for $q\in \dbN$ by taking, for prime numbers $p$ and non-negative integers $l$,
$$\kap(p^{4l})=p^{-l},\quad \kap(p^{4l+1})=4p^{-l-1/2},\quad \kap(p^{4l+2})=p^{-l-1},\quad 
\kap(p^{4l+3})=p^{-l-1}.$$
Then as a consequence of \cite[Theorem 4.1 and Lemmata 4.3, 4.4 and 6.2]{Vau1997},
\begin{align*}
g(a_1S\bet)&\ll \kap(q)P(1+P^4|\bet -a/q|)^{-1}+q^{1/2+\eps}\\
&\ll \kap(q)P(1+P^4|\bet -a/q|)^{-1/2}.
\end{align*}
We therefore deduce from (\ref{6.3}) that
\begin{equation}\label{6.4}
J(\lam)\ll \sum_{1\le q\le SP^{1/2}}(P\kap(q))^{9/4}\sum_{a=1}^q\int_{-\infty}^\infty 
\frac{|h(B(a/q+\gam)+\lam)|^4}{(1+P^4|\gam|)^{9/8}}\d\gam .
\end{equation}

\par By orthogonality, we find that
\begin{equation}\label{6.5}
\sum_{a=1}^q|h(B(a/q+\gam)+\lam)|^4\le q\sum_{\substack{1\le x_1,\ldots ,x_4\le P\\ 
q|B(x_1^4+x_2^4-x_3^4-x_4^4)}}1\le |B|^4(Pq^{-1}+1)^4q\rho(q),
\end{equation}
where $\rho(q)$ denotes the number of solutions of the congruence
$$x_1^4+x_2^4\equiv x_3^4+x_4^4\mmod{q},$$
with $1\le x_i\le q$ $(1\le i\le 4)$. The argument of \cite{VW2000} leading to equation (5.8) of that paper shows that
$$q\rho(q)\ll q^4\sum_{r|q}r\kap(r)^4.$$
Hence, on substituting (\ref{6.5}) into (\ref{6.4}), we obtain
$$J(\lam)\ll P^{25/4}\sum_{1\le q\le SP^{1/2}}\kap(q)^{9/4}\sum_{r|q}r\kap(r)^4\int_{-\infty}^\infty 
(1+P^4|\gam|)^{-9/8}\d\gam .$$
Observe that $\kap(q)^{9/4}\le q^{-1/20}\kap(q)^2$, and hence the argument completing the proof of 
\cite[Lemma 5.4]{VW2000} shows that for a suitable positive constant $C$, one has
$$J(\lam)\ll P^{9/4}\sum_{1\le q\le P}\kap(q)^{9/4}\sum_{r|q}r\kap(r)^4\ll P^{9/4}\prod_p 
\Bigl( 1+C\sum_{h=1}^\infty p^{-1-h/20}\Bigr).$$
Thus we obtain $J(\lam)\ll P^{9/4}$, and the proof of the lemma is complete.
\end{proof}

When $\grB\subseteq [0,1)^2$ is measurable, define the auxiliary mean value
$$U(\grB)=\iint_\grB |g(\Lam_1)H_0(\alp,\bet)|\d\alp\d\bet .$$

\begin{lemma}\label{lemma6.3} One has $U(\grn)\ll P^{14}(\log \log P)^{-1}$.
\end{lemma}

\begin{proof} Consider the auxiliary sets
$$\gre=\{ (\alp,\bet)\in \grn:\alp\in \grm\}\quad \text{and}\quad \grE=\{ (\alp,\bet)\in \grn:\alp\in \grM\}.$$
The treatment of the set $\gre$ is straightforward. On recalling that $\Lam_1=a_1\alp$, one finds via an enhanced version of
 Weyl's inequality (see \cite[Lemma 3]{Vau1986b}) that
$$\sup_{(\alp,\bet)\in \gre}|g(\Lam_1)|=\sup_{\alp\in\grm}|g(a_1\alp)|\ll P^{7/8+\eps}.$$
Then from Lemma \ref{lemma6.1}, one deduces that
\begin{equation}\label{6.6}
U(\gre)\ll P^{7/8+\eps}\int_0^1\int_0^1|H_0(\alp,\bet)|\d\alp\d\bet \ll P^{\frac{111}{8}+\frac{1}{2}\Del_{10}+\eps}
\ll P^{14}Q^{-1}.
\end{equation}

\par We turn next to the complementary set $\grE$, handling this via Lemma \ref{lemma6.2}. By applying the argument
 underlying the proof of Lemma \ref{lemma6.1}, as in the discussion following the statement of \cite[Lemma 10]{BW2007}, 
one finds that for some indices $l,m,n$ with $2\le l<m<n\le 22$, one has
$$U(\grE)\ll \iint_\grE g_1h_l^7h_m^7h_n^7\d\alp\d\bet .$$
Here, we have abbreviated $|g(\Lam_1)|$ to $g_1$. By relabelling variables if necessary, there is no loss of generality in
 supposing that $(l,m,n)=(2,3,4)$ and that $\Lam_3$ and $\Lam_4$ are each pairwise linearly independent of $\Lam_1$. 
Recall also that $\Lam_2,\Lam_3,\Lam_4$ are assumed to be pairwise linearly independent. Define
$$U_{ij}=\iint_\grE g_1^{9/4}h_i^4h_j^{12}\d\alp\d\bet \quad \text{and}\quad 
V_{ij}=\iint_\grE h_i^{12}h_j^{12}\d\alp\d\bet .$$
Then an application of H\"older's inequality yields the bound
\begin{equation}\label{6.7}
U(\grE)\ll \Bigl( \sup_{(\alp,\bet)\in \grE}|h(\Lam_2)h(\Lam_3)h(\Lam_4)|\Bigr)^{5/27}(U_{23}U_{24})^{2/9}
(V_{23}V_{24})^{17/81}V_{34}^{11/81}.
\end{equation}

\par Let $(i,j)$ be either $(2,3)$ or $(2,4)$. Recall that $\Lam_1=a_1\alp$, and change variables from $\bet$ to $\gam$ via
 the linear transformation $a_j\alp+b_j\bet=b_j\gam$. Note here that since $\Lam_1$ and $\Lam_j$ are inequivalent, then
 necessarily $b_j\ne 0$. Write $A=a_i-b_ia_j/b_j$ and recall (\ref{6.2}). Then in view of the definition of $\grE$, we may
 make use of the periodicity of the integrand to deduce that
$$U_{ij}\le \int_0^1\int_\grM |g(\Lam_1)|^{9/4}|h(A\alp+b_i\gam)|^4|h(b_j\gam)|^{12}\d\alp\d\gam \le 
W\sup_{\lam\in \dbR}J(\lam;A),$$
where
$$W=\int_0^1|h(b_j\gam)|^{12}\d\gam .$$
An application of Lemma \ref{lemma2.1} shows, via a change of variable, that $W=O(P^8)$, and so we deduce from Lemma
 \ref{lemma6.2} that $U_{23}U_{24}\ll (P^8)^2(P^{9/4})^2=P^{41/2}$. Since $\Lam_2$, $\Lam_3$, $\Lam_4$ are
 pairwise linearly independent, when $2\le i<j\le 4$ further changes of variable lead from Lemma \ref{lemma2.1} to the
 estimate
$$V_{ij}\ll \int_0^1\int_0^1|h(\tet_1)h(\tet_2)|^{12}\d\tet_1\d\tet_2=\Bigl( \int_0^1|h(\tet)|^{12}\d\tet\Bigr)^2
\ll (P^8)^2.$$
By substituting these estimates into (\ref{6.7}), we deduce thus far that
\begin{equation}\label{6.8}
U(\grE)\ll P^{14}\Bigl( P^{-3}\sup_{(\alp,\bet)\in \grE}|h(\Lam_2)h(\Lam_3)h(\Lam_4)|\Bigr)^{5/27}.
\end{equation}

\par Our final task is to bound the second factor on the right hand side of (\ref{6.8}). When $\tet$ is a real number with 
$|h(\tet)|\ge PQ^{-1/100}$, it follows from \cite[Lemma 2.1]{Woo2003} that there exist $a\in \dbZ$ and $q\in \dbN$ with
 $1\le q\le Q^{1/10}$, $(a,q)=1$ and $|q\tet -a|\le Q^{1/10}P^{-4}$. Consequently, if $\Lam_k$ and $\Lam_l$ are 
inequivalent linear forms, and $h_kh_l\ge P^2Q^{-1/100}$, then for $\sig=k,l$ there exist integers $d_\sig$ and $q_\sig$
 with
$$1\le q_\sig\le Q^{1/10},\quad (d_\sig,q_\sig)=1\quad \text{and}\quad |\Lam_\sig -d_\sig/q_\sig|
\le q_\sig^{-1}Q^{1/10}P^{-4}.$$
From here it follows as in the proof of \cite[Lemma 10]{BW2007} that $(\alp,\bet)\in \grN$. Thus
$$\sup_{(\alp,\bet)\in \grn}|h(\Lam_k)h(\Lam_l)|\ll P^2Q^{-1/100},$$
whence we obtain the estimate
$$\sup_{(\alp,\bet)\in \grE}|h(\Lam_2)h(\Lam_3)h(\Lam_4)|\ll P^3Q^{-1/100}.$$
On substituting into (\ref{6.8}), we conclude that $U(\grE)\ll P^{14}Q^{-1/600}$. The conclusion of the lemma therefore
 follows from (\ref{6.6}) on recalling that $\grn=\gre\cup \grE$.
\end{proof}

A trivial estimate for the generating function $h(\tet)$ now leads from the conclusion of Lemma \ref{lemma6.3} to the estimate
 (\ref{5.5}) by means of the relation
$$\iint_\grn g(\Lam_1)H(\alp,\bet)\d\alp\d\bet \ll P^{s-22}U(\grn).$$

\section{The major arcs analysis}
Not only is the analysis of the major arcs largely standard, but it is also very similar to the work in \cite[\S7]{BW2007}. A brief 
sketch of the analysis therefore suffices on the present occasion. We begin with some additional notation. Define
$$S(q,a)=\sum_{r=1}^qe(ar^4/q),\quad T(q,c,d)=q^{-s}\prod_{j=1}^sS(q,a_jc+b_jd),$$
$$\grA(q)=\underset{(a,b,q)=1}{\sum_{a=1}^q\sum_{b=1}^q}T(q,a,b)\quad \text{and}\quad \grS(X)=
\sum_{1\le q\le Q}\grA(q).$$
Also, with $\lam_j$ as shorthand for $a_j\xi+b_j\zet$, put
$$v(\tet)=\int_0^Pe(\tet \gam^4)\d\gam ,\quad w(\tet)=\int_{\del P}^Pe(\tet \gam^4)\d\gam ,\quad 
V(\xi,\zet)=w(a_1\xi)\prod_{j=2}^sv(\lam_j),$$
and writing $\calB(X)=[-XP^{-4},XP^{-4}]^2$, define
$$\grJ(X)=\iint_{\calB(X)}V(\xi,\zet)\d\xi\d\zet.$$

\par Recall that $\Lam_1=a_1\alp$. Then, as adapted to the current context, the argument leading to 
\cite[equation (7.5)]{BW2007} shows that there is a positive number $\rho$ having the property that whenever 
$(\alp,\bet)\in \grN(q,a,b)\subseteq \grN$, one has
$$g(\Lam_1)H(\alp,\bet)-\rho T(q,a,b)V(\alp-a/q,\bet-b/q)\ll P^s(\log P)^{-1/2}.$$
Integrating over $\grN$, we infer that
\begin{equation}\label{7.1}
\iint_\grN g(\Lam_1)H(\alp,\bet)\d\alp\d\bet =\rho \grS(Q)\grJ(Q)+O(P^{s-8}(\log P)^{-1/4}).
\end{equation}

\begin{lemma}\label{lemma7.1} Under the hypotheses of Theorem \ref{theorem1.1}, the limit
 $\grS=\underset{X\rightarrow \infty}\lim\grS(X)$ exists, and one has $\grS-\grS(X)\ll X^{-1}$. When the pair of equations 
(\ref{1.1}) has a non-singular $p$-adic solution for all primes $p$, moreover, one has $\grS\gg 1$.
\end{lemma}

\begin{proof} This lemma is an adaptation of \cite[Lemma 12]{BW2007} to our needs. We establish the estimate 
$\grA(q)=O(q^{-2})$ to replace the cognate bound \cite[equation (7.14)]{BW2007}. Once this is confirmed, all the
 conclusions drawn in the lemma follow just as in the aforementioned work \cite{BW2007}, and thus we may omit the details. 
We immitate the argument on \cite[page 890]{BW2007} to establish the aforementioned bound. Note first that we may suppose
 that whenever $(c,d)\in \dbZ^2\setminus\{(0,0)\}$, then the linear form $cL_1(\bftet)+dL_2(\bftet)$ contains at least $s-7$
 non-zero coefficients. Let $u_j=(q,ca_j+db_j)$ and apply \cite[Theorem 4.2]{Vau1997} to infer that
$$T(q,c,d)\ll q^{-11/2}(u_1u_2\ldots u_{22})^{1/4}.$$
Here, we have used the prearrangement of indices and removed dependence on potential indices with $22<j\le s$ by the use 
of trivial estimates. Recall now the multiplicities $r_1,\ldots ,r_t$ associated to the equivalence classes of the forms $\Lam_j$.
 Following the analysis on \cite[page 890]{BW2007}, one finds that there is a natural number $\Del$ depending only on 
$a_j,b_j$ $(1\le j\le s)$ such that
$$\grA(q)\ll q^{-7/2}\sum_{\substack{v_1,\ldots ,v_t\\ v_1v_2\ldots v_t|\Del q}}v_1^{(r_1-4)/4}\ldots v_t^{(r_t-4)/4}.$$
The upper bound $r_l\le 7$ therefore leads to the estimate $\grA(q)\ll q^{\eps-11/4}$, so that in view of our earlier comments, the proof of the lemma is complete.
\end{proof}

\begin{lemma}\label{lemma7.2} Under the hypotheses of Theorem \ref{theorem1.1}, the limit 
$\grJ=\underset{X\rightarrow \infty }\lim \grJ(X)$ exists, and one has $\grJ-\grJ(X)\ll P^{s-8}X^{-1}$. When the pair of
 equations (\ref{1.1}) has a non-singular real solution, moreover, one has $\grJ\gg P^{s-8}$.
\end{lemma}

\begin{proof} Write $\calBhat(X)$ for $\dbR^2\setminus \calB(X)$, and recall the prearrangement of indices introduced in \S5.
 Then a direct modification of the argument on \cite[page 891]{BW2007} confirms that, for a suitable positive number 
$\Tet=\Tet(\bfa,\bfb)$, one has
$$\underset{\calBhat(X)}\iint |v(\lam_2)\ldots v(\lam_{22})|\d\xi \d\zet \ll P^{21}\underset{\calBhat(\Tet X)}\iint 
(1+P^4|\xi|)^{-21/8}(1+P^4|\zet|)^{-21/8}\d\xi\d\zet .$$
By applying trivial bounds for $w(a_1\xi)$ and the additional factors $v(\lam_j)$ for $j>22$, we therefore conclude that
$$\iint_{\dbR^2\setminus \calB(X)}V(\xi,\zet)\d\xi \d\zet \ll P^{s-21}(P^{13}X^{-1})\ll P^{s-8}X^{-1}.$$
This bound replaces \cite[equation (7.18)]{BW2007}, and the lemma now follows in the same manner as Lemma 13 was 
proved in \cite{BW2007}.
\end{proof}

The conclusions of Lemmata \ref{lemma7.1} and \ref{lemma7.2} now combine with the asymptotic formula (\ref{7.1}) to deliver the relation
$$\iint_\grN g(\Lam_1)H(\alp,\bet)\d\alp\d\bet =\rho \grS \grJ +O(P^{s-8}Q^{-1})\gg P^{s-8},$$
thereby confirming the lower bound (\ref{5.4}). In view of the discussion concluding \S5, the lower bound $\calN(P)\gg P^{s-8}$ that establishes Theorem \ref{theorem1.1} now follows.

\bibliographystyle{amsbracket}
\providecommand{\bysame}{\leavevmode\hbox to3em{\hrulefill}\thinspace}

\end{document}